\documentclass[reqno]{amsart}
\usepackage{amsfonts}

\newtheorem{theorem}{Theorem}
\theoremstyle{plain}

\newtheorem{corollary}{Corollary}

\newtheorem{lemma}{Lemma}

\newtheorem{problem}{Problem}
\newtheorem{proposition}{Proposition}

\numberwithin{equation}{section}
\numberwithin{theorem}{section}
\numberwithin{algorithm}{section}
\numberwithin{axiom}{section}
\numberwithin{case}{section}
\numberwithin{claim}{section}
\numberwithin{conclusion}{section}
\numberwithin{condition}{section}
\numberwithin{conjecture}{section}
\numberwithin{corollary}{section}
\numberwithin{criterion}{section}
\numberwithin{definition}{section}
\numberwithin{example}{section}
\numberwithin{exercise}{section}
\numberwithin{lemma}{section}
\numberwithin{notation}{section}
\numberwithin{problem}{section}
\numberwithin{proposition}{section}
\numberwithin{remark}{section}
\numberwithin{solution}{section}

\input{tcilatex}

\begin{document}
\title[Manifolds with positive Yamabe invariants and Paneitz operator]{%
Riemannian manifolds with positive Yamabe invariant and Paneitz operator}
\author{Matthew J. Gursky}
\address{Department of Mathematics, University of Notre Dame, Notre Dame, IN
46556}
\email{mgursky@nd.edu}
\author{Fengbo Hang}
\address{Courant Institute, New York University, 251 Mercer Street, New York
NY 10012}
\email{fengbo@cims.nyu.edu}
\author{Yueh-Ju Lin}
\address{Department of Mathematics, University of Michigan, Ann Arbor, MI
48109}
\email{yuehjul@umich.edu}

\begin{abstract}
For a Riemannian manifold with dimension at least six, we prove that the
existence of a conformal metric with positive scalar and Q curvature is
equivalent to the positivity of both the Yamabe invariant and the Paneitz
operator.
\end{abstract}

\maketitle

\section{Introduction\label{sec1}}

Let $\left( M,g\right) $ be a smooth compact Riemannian manifold with
dimension $n\geq 3$. Denote $\left[ g\right] $ as the conformal class of
metrics associated with $g$. The Yamabe invariant is given by (see \cite{LP})%
\begin{equation}
Y\left( M,g\right) =\inf_{\widetilde{g}\in \left[ g\right] }\frac{\int_{M}%
\widetilde{R}d\widetilde{\mu }}{\widetilde{\mu }\left( M\right) ^{\frac{n-2}{%
n}}},  \label{eq1.1}
\end{equation}%
here $\widetilde{R}$ is the scalar curvature of $\widetilde{g}$ and $%
\widetilde{\mu }$ is the measure associated with $\widetilde{g}$. In terms
of the conformal Laplacian operator%
\begin{equation}
L=-\frac{4\left( n-1\right) }{n-2}\Delta +R,  \label{eq1.2}
\end{equation}%
we have%
\begin{eqnarray}
Y\left( M,g\right) &=&\inf_{\substack{ u\in C^{\infty }\left( M\right)  \\ %
u>0}}\frac{\int_{M}Lu\cdot ud\mu }{\left\Vert u\right\Vert _{L^{\frac{2n}{n-2%
}}}^{2}}  \label{eq1.3} \\
&=&\inf_{u\in H^{1}\left( M\right) \backslash \left\{ 0\right\} }\frac{%
\int_{M}\left( \frac{4\left( n-1\right) }{n-2}\left\vert \nabla u\right\vert
^{2}+Ru^{2}\right) d\mu }{\left\Vert u\right\Vert _{L^{\frac{2n}{n-2}}}^{2}}.
\notag
\end{eqnarray}%
In particular $Y\left( M,g\right) >0$ if and only if the first eigenvalue $%
\lambda _{1}\left( L\right) >0$. Moreover based on the fact that the first
eigenfunction of $L$ must be strictly positive or negative, we know%
\begin{equation}
\lambda _{1}\left( L\right) >0\Leftrightarrow \text{there exists a }%
\widetilde{g}\in \left[ g\right] \text{ with }\widetilde{R}>0.  \label{eq1.4}
\end{equation}%
Here we are looking for similar characterization for Paneitz operator and $Q$
curvature (see \cite{B,P}). More precisely we are interested in the solution
to

\begin{problem}
\label{problem1.1}For a Riemannian manifold with dimension at least five,
can we find a conformal invariant condition which is equivalent to the
existence of a conformal metric with positive scalar and $Q$ curvature?
\end{problem}

In view of the results on positivity of Paneitz operator in \cite{GM,XY},
many people suspect the conformal invariant condition wanted in Problem \ref%
{problem1.1} should be the positivity of Yamabe invariant and Paneitz
operator. As a consequence of the main result below, this is verified for
dimension at least six (see Corollary \ref{cor1.1}). It is very likely the
statement is still true for dimension five. However due to technical
constrains in our approach, dimension five case remains an open problem.

To write down the formula of $Q$ curvature and Paneitz operator, following 
\cite{B}, let%
\begin{equation}
J=\frac{R}{2\left( n-1\right) },\quad A=\frac{1}{n-2}\left( Rc-Jg\right) ,
\label{eq1.5}
\end{equation}%
here $Rc$ is the Ricci curvature tensor. The $Q$ curvature is given by%
\begin{equation}
Q=-\Delta J-2\left\vert A\right\vert ^{2}+\frac{n}{2}J^{2}.  \label{eq1.6}
\end{equation}%
The Paneitz operator is given by%
\begin{equation}
P\varphi =\Delta ^{2}\varphi +\func{div}\left( 4A\left( \nabla \varphi
,e_{i}\right) e_{i}-\left( n-2\right) J\nabla \varphi \right) +\frac{n-4}{2}%
Q\varphi .  \label{eq1.7}
\end{equation}%
Here $e_{1},\cdots ,e_{n}$ is a local orthonormal frame with respect to $g$.
For $n\geq 5$, under a conformal change of metric, the Paneitz operator
satisfies 
\begin{equation}
P_{\rho ^{\frac{4}{n-4}}g}\varphi =\rho ^{-\frac{n+4}{n-4}}P_{g}\left( \rho
\varphi \right) ,  \label{eq1.8}
\end{equation}%
(see \cite{B}); compare to the conformal covariant property of the conformal
Laplacian (\ref{eq1.2}). In addition, the $Q$ curvature is transformed by
the formula%
\begin{equation}
Q_{\rho ^{\frac{4}{n-4}}g}=\frac{2}{n-4}\rho ^{-\frac{n+4}{n-4}}P_{g}\rho .
\label{eq1.9}
\end{equation}

We now define two conformal invariants related to the $Q$-curvature. First,
in analogy with the Yamabe invariant, we define%
\begin{equation}
Y_{4}^{+}\left( M,g\right) =\frac{n-4}{2}\inf_{\widetilde{g}\in \left[ g%
\right] }\frac{\int_{M}\widetilde{Q}d\widetilde{\mu }}{\left( \widetilde{\mu 
}\left( M\right) \right) ^{\frac{n-4}{n}}}=\inf_{\substack{ u\in C^{\infty
}\left( M\right)  \\ u>0}}\frac{\int_{M}Pu\cdot ud\mu }{\left\Vert
u\right\Vert _{L^{\frac{2n}{n-4}}}^{2}}.  \label{eq1.10}
\end{equation}%
We use $Y_{4}^{+}\left( M,g\right) $ to emphasize the infimum is taken over
positive functions (i.e. conformal factors). To define the second invariant,
denote

\begin{eqnarray}
E\left( \varphi \right) &=&\int_{M}P\varphi \cdot \varphi d\mu
\label{eq1.11} \\
&=&\int_{M}\left( \left( \Delta \varphi \right) ^{2}-4A\left( \nabla \varphi
,\nabla \varphi \right) +\left( n-2\right) J\left\vert \nabla \varphi
\right\vert ^{2}+\frac{n-4}{2}Q\varphi ^{2}\right) d\mu .  \notag
\end{eqnarray}%
It is clear that $E\left( \varphi \right) $ is well defined for $\varphi \in
H^{2}\left( M\right) $. Define%
\begin{equation}
Y_{4}\left( M,g\right) =\inf_{u\in H^{2}\left( M\right) \backslash \left\{
0\right\} }\frac{E\left( u\right) }{\left\Vert u\right\Vert _{L^{\frac{2n}{%
n-4}}}^{2}}.  \label{eq1.12}
\end{equation}%
Clearly%
\begin{equation*}
Y_{4}\left( M,g\right) \leq Y_{4}^{+}\left( M,g\right) ,
\end{equation*}
but in general (and in contrast to the usual Yamabe invariant) we have an
inequality instead of equality, due to the fact Paneitz operator is fourth
order. Note that from standard elliptic theory $Y_{4}\left( M,g\right) >0$
if and only if the first eigenvalue $\lambda _{1}\left( P\right) >0$ i.e. $P$
is positive definite.

When the Yamabe invariant $Y\left( M,g\right) >0$, there is another closely
related quantity $Y_{4}^{\ast }\left( M,g\right) $ defined by%
\begin{equation}
Y_{4}^{\ast }\left( M,g\right) =\frac{n-4}{2}\inf_{\substack{ \widetilde{g}%
\in \left[ g\right]  \\ \widetilde{R}>0}}\frac{\int_{M}\widetilde{Q}d%
\widetilde{\mu }}{\left( \widetilde{\mu }\left( M\right) \right) ^{\frac{n-4%
}{n}}}.  \label{eq1.13}
\end{equation}%
Obviously%
\begin{equation}
Y_{4}\left( M,g\right) \leq Y_{4}^{+}\left( M,g\right) \leq Y_{4}^{\ast
}\left( M,g\right) .  \label{eq1.14}
\end{equation}
One of the main goals of this paper is to understand the relationship
between these quantities, and their connection to the existence of a metric
with positive scalar and $Q$-curvature.

Our motivation for studying these problems comes from recent progress in the
understanding of $Q$ curvature equations for dimension at least five in \cite%
{GM,HY1,HY2}. Paneitz operator and $Q$ curvature were brought to attention
in \cite{CGY}. For dimension at least five, in various work \cite{HeR,HuR,R}%
, people had realized the important role of positivity of Green's function
of Paneitz operator in understanding the $Q$ curvature equations. Such kind
of positivity is hard to get due to the lack of maximum principle for fourth
order equations. A breakthrough was achieved in \cite{GM}, namely for $n\geq
5$, if $R>0$ and $Q>0$, then $P>0$ and the Green's function of Paneitz
operator $G_{P}$ is strictly positive. Subsequently in \cite{HY1}, for
positive Yamabe invariant case it was found the positivity of Green's
function of Paneitz operator is equivalent to the existence of a conformal
metric with positive $Q$ curvature. Indeed it was shown that if $n\geq 5$, $%
Y\left( M,g\right) >0$, then there exists a conformal metric with positive $%
Q $ curvature if and only if $\ker P=0$ and $G_{P}>0$, it is also equivalent
to $\ker P=0$ and $G_{P,p}>0$ for a fixed $p$. Note the positivity of
Green's function is a conformal invariant condition. In \cite{GM}, it was
shown that $R>0$ and $Q>0$ implies $Y_{4}\left( M,g\right) $ is achieved at
a positive smooth function. The assumption was relaxed to $Y\left(
M,g\right) >0,Q>0$ and $P>0$ in \cite{HY2}. Trying to understand relations
between various assumptions motivates problems considered here.

\begin{theorem}
\label{thm1.1}Let $\left( M,g\right) $ be a smooth compact Riemannian
manifold with dimension $n\geq 6$. If $Y\left( M,g\right) >0$ and $%
Y_{4}^{\ast }\left( M,g\right) >0$, then there exists a metric $\widetilde{g}%
\in \left[ g\right] $ satisfying $\widetilde{R}>0$ and $\widetilde{Q}>0$.
\end{theorem}

Combine Theorem \ref{thm1.1} with existence and positivity results in \cite%
{GM,HY2,XY} we have the following corollaries.

\begin{corollary}
\label{cor1.1}Let $\left( M,g\right) $ be a smooth compact Riemannian
manifold with dimension $n\geq 6$. Then the following statements are
equivalent

\begin{enumerate}
\item $Y\left( M,g\right) >0,P>0$.

\item $Y\left( M,g\right) >0,Y_{4}^{\ast }\left( M,g\right) >0$.

\item there exists a metric $\widetilde{g}\in \left[ g\right] $ satisfying $%
\widetilde{R}>0$ and $\widetilde{Q}>0$.
\end{enumerate}
\end{corollary}

Corollary \ref{cor1.1} answers Problem \ref{problem1.1} for dimension at
least six.

\begin{corollary}
\label{cor1.2}Let $\left( M,g\right) $ be a smooth compact Riemannian
manifold with dimension $n\geq 6$. If $Y\left( M,g\right) >0$ and $%
Y_{4}^{\ast }\left( M,g\right) >0$, then $P>0$, the Green's function $%
G_{P}>0 $, and $Y_{4}\left( M,g\right) $ is achieved at a positive smooth
function $u $ with $R_{u^{\frac{4}{n-4}}g}>0$ and $Q_{u^{\frac{4}{n-4}%
}g}=const$. In particular,%
\begin{equation*}
Y_{4}\left( M,g\right) =Y_{4}^{+}\left( M,g\right) =Y_{4}^{\ast }\left(
M,g\right) .
\end{equation*}
\end{corollary}

The dimensional restriction $n\geq 6$ is an unfortunate by-product of our
technique and it is very likely the result holds in dimension five as well.
To explain our approach, we first point out that the $Q$ curvature equation
is variational: a metric has constant $Q$ curvature if and only if it is a
critical point of the total $Q$ curvature functional $\int_{M}Q_{g}d\mu _{g}$
with $g$ running through the set of conformal metrics with unit volume. A
closely related quantity is%
\begin{equation}
\sigma _{2}\left( A\right) =\frac{1}{2}\left( J^{2}-\left\vert A\right\vert
^{2}\right) .  \label{eq1.15}
\end{equation}%
Indeed $\sigma _{2}\left( A\right) $ equation is also variational in similar
sense. Since%
\begin{equation}
Q=-\Delta J+\frac{n-4}{2}J^{2}+4\sigma _{2}\left( A\right) ,  \label{eq1.16}
\end{equation}%
we have%
\begin{equation}
\int_{M}Qd\mu =\frac{n-4}{2}\int_{M}J^{2}d\mu +4\int_{M}\sigma _{2}\left(
A\right) d\mu .  \label{eq1.17}
\end{equation}%
Obviously $\int_{M}J^{2}d\mu $ is always nonnegative. For $t\geq 1$ consider
the functional%
\begin{equation}
\frac{n-4}{2}t\int_{M}J_{g}^{2}d\mu _{g}+4\int_{M}\sigma _{2}\left(
A_{g}\right) d\mu _{g}.  \label{eq1.18}
\end{equation}%
A critical metric of this functional restricted to the space of conformal
metrics of unit volume satisfies%
\begin{equation}
t\left( -\Delta J+\frac{n-4}{2}J^{2}\right) +4\sigma _{2}\left( A\right)
=const.  \label{eq1.19}
\end{equation}

In the appendix, we will use elementary methods to show that if $Y\left(
M,g\right) >0$, then there exists $g_{0}\in \left[ g\right] $, $g_{0}=u_{0}^{%
\frac{4}{n-4}}g$ and $t_{0}\gg 1$ such that%
\begin{equation}
t_{0}\left( -\Delta _{0}J_{0}+\frac{n-4}{2}J_{0}^{2}\right) +4\sigma
_{2}\left( A_{0}\right) >0  \label{eq1.20}
\end{equation}%
and $J_{0}>0$. Define $f$ as%
\begin{equation}
t_{0}\left( -\Delta _{0}J_{0}+\frac{n-4}{2}J_{0}^{2}\right) +4\sigma
_{2}\left( A_{0}\right) =fu_{0}^{-\frac{n+4}{n-4}}.  \label{eq1.21}
\end{equation}%
Then for $1\leq t\leq t_{0}$ we consider the following 1-parameter family of
equations:%
\begin{equation}
t\left( -\widetilde{\Delta }J+\frac{n-4}{2}\widetilde{J}^{2}\right) +4\sigma
_{2}\left( \widetilde{A}\right) =fu^{-\frac{n+4}{n-4}},\quad \widetilde{g}%
=u^{\frac{4}{n-4}}g.  \label{eq1.22}
\end{equation}%
Let%
\begin{eqnarray}
&&S  \label{eq1.23} \\
&=&\left\{ t\in \left[ 1,t_{0}\right] :\text{there exists positive smooth }u%
\text{ solving (\ref{eq1.22}), with }\widetilde{R}>0\right\} .  \notag
\end{eqnarray}%
Then $t_{0}\in S$. We will show $S$ is both open and closed by the implicit
function theorem and apriori estimates. Hence $S=\left[ 1,t_{0}\right] $. It
follows that there exists a $\widetilde{g}\in \left[ g\right] $ with $%
\widetilde{R}>0$ and $\widetilde{Q}>0$.

We conclude the introduction with some remarks. The dimensional restriction $%
n\geq 6$ appears in both the open and closed part of the argument. The power
of the conformal factor $u$ on the right hand side of (\ref{eq1.22}) is
chosen to be negative to give better estimate of solutions. Moreover this
choice also leads to a good sign of the zeroth order term in the linearized
operator. We also observe that our path of equations is variational, it has
a divergence structure which we will exploit in apriori estimates. At last
we note that in dimension four a path of equations which is analogous to (%
\ref{eq1.22}) is considered in \cite{CGY} to produce a conformal metric with
positive scalar curvature and $\sigma _{2}\left( A\right) $ curvature
assuming the positivity of Yamabe invariant and conformal invariant $%
\int_{M}\sigma _{2}\left( A\right) d\mu $.

\textbf{Acknowledgements. }The first author is supported in part by the NSF
grant DMS--1206661. We would like to thank Professor Alice Chang and Paul
Yang for valuable discussions.

\section{The method of continuity and openness\label{sec2}}

In this section we set up the continuity method. It will be more convenient
if we first rewrite equation (\ref{eq1.22}) in terms of $Q$ and $\sigma
_{2}\left( A\right) $. Using (\ref{eq1.16}), equation (\ref{eq1.22}) can be
expressed as%
\begin{equation}
t\left( \widetilde{Q}-4\sigma _{2}\left( \widetilde{A}\right) \right)
+4\sigma _{2}\left( \widetilde{A}\right) =fu^{-\frac{n+4}{n-4}},
\label{eq2.1}
\end{equation}%
hence%
\begin{equation}
t\widetilde{Q}+4\left( 1-t\right) \sigma _{2}\left( \widetilde{A}\right)
=fu^{-\frac{n+4}{n-4}}.  \label{eq2.2}
\end{equation}%
Dividing by $t$ (recall $t\geq 1$) and denoting%
\begin{equation}
\lambda =\frac{4\left( t-1\right) }{t},\quad \chi =\frac{f}{t},
\label{eq2.3}
\end{equation}%
equation (\ref{eq1.22}) is equivalent to%
\begin{equation}
\widetilde{Q}-\lambda \sigma _{2}\left( \widetilde{A}\right) =\chi u^{-\frac{%
n+4}{n-4}},  \label{eq2.4}
\end{equation}%
here $\widetilde{g}=u^{\frac{4}{n-4}}g$ and $0\leq \lambda <4$.

To write this equation in terms of the conformal factor $u$, observe that
the Schouten tensor of the conformal metric $\widetilde{g}$ is given by%
\begin{equation}
\widetilde{A}_{ij}=A_{ij}-\frac{2}{n-4}u^{-1}u_{ij}-\frac{2}{\left(
n-4\right) ^{2}}u^{-2}\left\vert \nabla u\right\vert ^{2}g_{ij}+\frac{%
2\left( n-2\right) }{\left( n-4\right) ^{2}}u^{-2}u_{i}u_{j},  \label{eq2.5}
\end{equation}%
hence%
\begin{eqnarray}
&&\sigma _{2}\left( \widetilde{A}\right)  \label{eq2.6} \\
&=&u^{-\frac{8}{n-4}}\left[ \sigma _{2}\left( A\right) +\frac{2}{\left(
n-4\right) ^{2}}u^{-2}\left( \Delta u\right) ^{2}-\frac{2}{\left( n-4\right)
^{2}}u^{-2}\left\vert D^{2}u\right\vert ^{2}+\frac{4}{\left( n-4\right) ^{3}}%
u^{-3}\left\vert \nabla u\right\vert ^{2}\Delta u\right.  \notag \\
&&+\frac{4\left( n-2\right) }{\left( n-4\right) ^{3}}u^{-3}u_{ij}u_{i}u_{j}-%
\frac{2}{n-4}Ju^{-1}\Delta u+\frac{2}{n-4}u^{-1}A_{ij}u_{ij}-\frac{2\left(
n-1\right) }{\left( n-4\right) ^{3}}u^{-4}\left\vert \nabla u\right\vert ^{4}
\notag \\
&&\left. -\frac{2}{\left( n-4\right) ^{2}}Ju^{-2}\left\vert \nabla
u\right\vert ^{2}-\frac{2\left( n-2\right) }{\left( n-4\right) ^{2}}%
u^{-2}A_{ij}u_{i}u_{j}\right] .  \notag
\end{eqnarray}%
Using the formula (\ref{eq1.9}) we also have%
\begin{equation}
\widetilde{Q}=\frac{2}{n-4}u^{-\frac{n+4}{n-4}}Pu  \label{eq2.7}
\end{equation}%
Substituting (\ref{eq2.6}) and (\ref{eq2.7}) into (\ref{eq2.4}), then
multiplying through by $\frac{n-4}{2}u^{\frac{n+4}{n-4}}$ we have

\begin{eqnarray}
&&Pu  \label{eq2.8} \\
&=&\lambda u\left[ \frac{n-4}{2}\sigma _{2}\left( A\right) +\frac{1}{n-4}%
u^{-2}\left( \Delta u\right) ^{2}-\frac{1}{n-4}u^{-2}\left\vert
D^{2}u\right\vert ^{2}+\frac{2}{\left( n-4\right) ^{2}}u^{-3}\left\vert
\nabla u\right\vert ^{2}\Delta u\right.  \notag \\
&&+\frac{2\left( n-2\right) }{\left( n-4\right) ^{2}}%
u^{-3}u_{ij}u_{i}u_{j}-Ju^{-1}\Delta u+u^{-1}A_{ij}u_{ij}-\frac{n-1}{\left(
n-4\right) ^{2}}u^{-4}\left\vert \nabla u\right\vert ^{4}  \notag \\
&&\left. -\frac{1}{n-4}Ju^{-2}\left\vert \nabla u\right\vert ^{2}-\frac{n-2}{%
n-4}u^{-2}A_{ij}u_{i}u_{j}\right] +\frac{n-4}{2}\chi .  \notag
\end{eqnarray}%
Also, the condition that $\widetilde{J}>0$ is equivalent to the inequality%
\begin{equation}
\Delta u<-\frac{2}{n-4}u^{-1}\left\vert \nabla u\right\vert ^{2}+\frac{n-4}{2%
}Ju.  \label{eq2.9}
\end{equation}

We begin with a fact which permits us to start the continuity process.

\begin{proposition}
\label{prop2.1}Let $\left( M,g\right) $ be a smooth compact Riemannian
manifold with dimension $n\geq 5$. If $Y\left( M,g\right) >0$, then there
exists a metric $\widetilde{g}\in \left[ g\right] $ with%
\begin{equation}
-\widetilde{\Delta }\widetilde{J}+\frac{n-4}{2}\widetilde{J}^{2}>0,\quad 
\widetilde{J}>0.  \label{eq2.10}
\end{equation}
\end{proposition}

Note this proposition is clearly a consequence of the solution to the Yamabe
problem (\cite{LP}). In the Appendix we will provide an elementary proof. In
view of Proposition \ref{prop2.1} we can find $\widetilde{g}=u_{0}^{\frac{4}{%
n-4}}g$ such that%
\begin{equation}
-\widetilde{\Delta }\widetilde{J}+\frac{n-4}{2}\widetilde{J}^{2}>0,\quad 
\widetilde{J}>0.  \label{eq2.11}
\end{equation}%
Since%
\begin{equation}
\widetilde{Q}-\lambda \sigma _{2}\left( \widetilde{A}\right) =-\widetilde{%
\Delta }\widetilde{J}+\frac{n-4}{2}\widetilde{J}^{2}+\left( 4-\lambda
\right) \sigma _{2}\left( \widetilde{A}\right) ,  \label{eq2.12}
\end{equation}%
we can find $0<\lambda _{0}<4$ close enough to $4$ such that $\widetilde{Q}%
-\lambda _{0}\sigma _{2}\left( \widetilde{A}\right) >0$. In particular $%
\widetilde{g}=u_{0}^{\frac{4}{n-4}}g$ is a solution of (\ref{eq2.4}), with%
\begin{equation}
\chi =\left( \widetilde{Q}-\lambda _{0}\sigma _{2}\left( \widetilde{A}%
\right) \right) u_{0}^{\frac{n+4}{n-4}}  \label{eq2.13}
\end{equation}%
Define%
\begin{equation}
\Sigma =\left\{ 0\leq \lambda \leq \lambda _{0}:\exists u\in C^{\infty
}\left( M\right) ,u>0,\text{ satisfying (\ref{eq2.4}) and }\widetilde{J}%
>0\right\} .  \label{eq2.14}
\end{equation}%
It is clear that $\lambda _{0}\in \Sigma $. Indeed in this case $u_{0}$ is a
solution. On the other hand if we can show $0\in \Sigma $, then it follows
there exists a metric $\widetilde{g}\in \left[ g\right] $ with $\widetilde{J}%
>0$ and $\widetilde{Q}>0$. To achieve this we will show $\Sigma $ is both
open and closed.

To prove that $\Sigma $ is open, we consider the linearized operator. To
this end, define the map%
\begin{equation}
\widetilde{g}=u^{\frac{4}{n-4}}g\mapsto \mathcal{N}\left[ u\right] =%
\widetilde{Q}-\lambda \sigma _{2}\left( \widetilde{A}\right) -\chi u^{-\frac{%
n+4}{n-4}}.  \label{eq2.15}
\end{equation}%
Then $\mathcal{N}\left[ u\right] =0$ if and only if $\widetilde{g}$ is a
solution of (\ref{eq2.4}). Let $\widetilde{S}$ denote the linearization of $%
\mathcal{N}$ at $u$:%
\begin{equation}
\widetilde{S}\varphi =\left. \frac{d}{dt}\right\vert _{t=0}\mathcal{N}\left[
u+t\varphi \right] .  \label{eq2.16}
\end{equation}%
To compute $\widetilde{S}$ we use the standard formulas for the variation of
the $Q$ curvature and the Schouten tensor%
\begin{eqnarray}
\left. \frac{d}{dt}\right\vert _{t=0}Q_{\left( 1+t\psi \right) g} &=&\frac{1%
}{2}P_{g}\psi -\frac{n+4}{4}Q_{g}\psi ,  \label{eq2.17} \\
\left. \frac{d}{dt}\right\vert _{t=0}\sigma _{2}\left( A_{\left( 1+t\psi
\right) g}\right) &=&-\frac{1}{2}J_{g}\Delta _{g}\psi +\frac{1}{2}g\left(
A_{g},D_{g}^{2}\psi \right) -2\sigma _{2}\left( A_{g}\right) \psi .
\label{eq2.18}
\end{eqnarray}%
In our setting, using%
\begin{equation}
\left( u+t\varphi \right) ^{\frac{4}{n-4}}g=\left( 1+tu^{-1}\varphi \right)
^{\frac{4}{n-4}}\widetilde{g}  \label{eq2.19}
\end{equation}%
and%
\begin{equation}
\left. \frac{d}{dt}\right\vert _{t=0}\left( 1+tu^{-1}\varphi \right) ^{\frac{%
4}{n-4}}=\frac{4}{n-4}u^{-1}\varphi ,  \label{eq2.20}
\end{equation}%
we see

\begin{equation}
\widetilde{S}\varphi =\frac{2}{n-4}\widetilde{H}\left( u^{-1}\varphi \right)
,  \label{eq2.21}
\end{equation}%
where%
\begin{eqnarray}
&&\widetilde{H}\varphi  \label{eq2.22} \\
&=&\widetilde{P}\varphi -\frac{n+4}{2}\widetilde{Q}\varphi +\lambda \left( 
\widetilde{J}\widetilde{\Delta }\varphi -\widetilde{g}\left( \widetilde{A},%
\widetilde{D}^{2}\varphi \right) +4\sigma _{2}\left( \widetilde{A}\right)
\varphi \right) +\frac{n+4}{2}\chi u^{-\frac{n+4}{n-4}}\varphi .  \notag
\end{eqnarray}

The openness of $\Sigma $ follows from implicit function theorem and
standard elliptic theory, together with the following lemma:

\begin{lemma}
\label{lem2.1}Let $\left( M,g\right) $ be a smooth compact Riemannian
manifold with dimension $n\geq 6$. If $0\leq \lambda \leq 4,$%
\begin{equation}
Q-\lambda \sigma _{2}\left( A\right) >0,\quad J>0,  \label{eq2.23}
\end{equation}%
then the operator%
\begin{eqnarray}
&&H\varphi  \label{eq2.24} \\
&=&P\varphi -\frac{n+4}{2}Q\varphi +\lambda \left( J\Delta \varphi
-A_{ij}\varphi _{ij}+4\sigma _{2}\left( A\right) \varphi \right) +\frac{n+4}{%
2}\left( Q-\lambda \sigma _{2}\left( A\right) \right) \varphi  \notag
\end{eqnarray}%
is positive definite.
\end{lemma}

\begin{proof}
For any smooth function $\varphi $,%
\begin{eqnarray}
&&\int_{M}H\varphi \cdot \varphi d\mu  \label{eq2.25} \\
&=&\int_{M}\left[ \left( \Delta \varphi \right) ^{2}+\left( n-2-\lambda
\right) J\left\vert \nabla \varphi \right\vert ^{2}-\left( 4-\lambda \right)
A\left( \nabla \varphi ,\nabla \varphi \right) +\frac{n-4}{2}\left(
Q-\lambda \sigma _{2}\left( A\right) \right) \varphi ^{2}\right] d\mu . 
\notag
\end{eqnarray}%
Using the Bochner formula%
\begin{equation}
\int_{M}\left( \Delta \varphi \right) ^{2}d\mu =\int_{M}\left\vert
D^{2}\varphi \right\vert ^{2}d\mu +\int_{M}J\left\vert \nabla \varphi
\right\vert ^{2}d\mu +\left( n-2\right) \int_{M}A\left( \nabla \varphi
,\nabla \varphi \right) d\mu ,  \label{eq2.26}
\end{equation}%
we see%
\begin{eqnarray}
&&\int_{M}H\varphi \cdot \varphi d\mu  \label{eq2.27} \\
&=&\frac{n-6+\lambda }{n-2}\int_{M}\left( \Delta \varphi \right) ^{2}d\mu +%
\frac{4-\lambda }{n-2}\int_{M}\left\vert D^{2}\varphi \right\vert ^{2}d\mu 
\notag \\
&&+\frac{4-\lambda +\left( n-2\right) \left( n-2-\lambda \right) }{n-2}%
\int_{M}J\left\vert \nabla \varphi \right\vert ^{2}d\mu +\frac{n-4}{2}%
\int_{M}\left( Q-\lambda \sigma _{2}\left( A\right) \right) \varphi ^{2}d\mu
.  \notag
\end{eqnarray}%
Note for $n\geq 6$ and $0\leq \lambda \leq 4$, all coefficients before the
integral sign are nonnegative. As a consequence $H$ is positive definite.
\end{proof}

For $n=5$, note that the coefficient of $\int_{M}J\left\vert \nabla \varphi
\right\vert ^{2}d\mu $ is equal to $\frac{13-4\lambda }{3}$ and it is
negative when $\lambda $ is close to $4$.

\section{Apriori estimate\label{sec3}}

In this section we prove apriori estimate for smooth positive solutions to (%
\ref{eq2.4}) with positive scalar curvature for $0\leq \lambda \leq \lambda
_{0}$. An immediate consequence is that the set $\Sigma $ (see (\ref{eq2.14}%
)) is closed.

\begin{lemma}
\label{lem3.1}Assume $\left( M,g\right) $ is a smooth compact Riemannian
manifold with dimension $n\geq 6$. If $Y\left( M,g\right) >0$, $Y_{4}^{\ast
}\left( M,g\right) >0$ and $0\leq \lambda \leq \lambda _{0}<4$, then any
smooth positive solution $u$ to (\ref{eq2.4}) with $\widetilde{J}>0$
satisfies%
\begin{equation}
\left\Vert u\right\Vert _{L^{\frac{2n}{n-4}}}\leq c  \label{eq3.1}
\end{equation}%
and%
\begin{equation}
u\geq c>0.  \label{eq3.2}
\end{equation}%
Here $c$ is independent of $u$ and $\lambda $.
\end{lemma}

\begin{proof}
Let $\widetilde{g}=u^{\frac{4}{n-4}}g$. Using%
\begin{equation}
\int_{M}\widetilde{Q}d\widetilde{\mu }=4\int_{M}\sigma _{2}\left( \widetilde{%
A}\right) d\widetilde{\mu }+\frac{n-4}{2}\int_{M}\widetilde{J}^{2}d%
\widetilde{\mu },  \label{eq3.3}
\end{equation}%
and (\ref{eq2.4}) we get%
\begin{equation}
\left( 1-\frac{\lambda }{4}\right) \int_{M}\widetilde{Q}d\widetilde{\mu }+%
\frac{n-4}{8}\lambda \int_{M}\widetilde{J}^{2}d\widetilde{\mu }=\int_{M}\chi
u^{-\frac{n+4}{n-4}}d\widetilde{\mu }=\int_{M}\chi ud\mu .  \label{eq3.4}
\end{equation}%
By the definition of $Y_{4}^{\ast }\left( M,g\right) $ (see (\ref{eq1.13}))
we have%
\begin{eqnarray}
\left\Vert u\right\Vert _{L^{\frac{2n}{n-4}}}^{2} &=&\widetilde{\mu }\left(
M\right) ^{\frac{n-4}{n}}  \label{eq3.5} \\
&\leq &\frac{n-4}{2}\frac{1}{Y_{4}^{\ast }\left( M,g\right) }\int_{M}%
\widetilde{Q}d\widetilde{\mu }  \notag \\
&\leq &c\int_{M}\chi ud\mu  \notag \\
&\leq &c\left\Vert u\right\Vert _{L^{\frac{2n}{n-4}}}.  \notag
\end{eqnarray}%
Hence%
\begin{equation}
\left\Vert u\right\Vert _{L^{\frac{2n}{n-4}}}\leq c.  \label{eq3.6}
\end{equation}

Multiplying both sides of equation (\ref{eq2.8}) by $u^{\alpha }$ and doing
integration by parts we get%
\begin{eqnarray}
&&\frac{n-4}{2}\int_{M}\chi u^{\alpha }d\mu  \label{eq3.7} \\
&=&\alpha \int_{M}u^{\alpha -1}\left( \Delta u\right) ^{2}d\mu +\left[
\alpha \left( \alpha -1\right) +\frac{3\alpha -1}{2\left( n-4\right) }%
\lambda \right] \int_{M}u^{\alpha -2}\left\vert \nabla u\right\vert
^{2}\Delta ud\mu  \notag \\
&&+\frac{\left( n-4\right) \alpha ^{2}-\left( n-8\right) \alpha -2}{2\left(
n-4\right) ^{2}}\lambda \int_{M}u^{\alpha -3}\left\vert \nabla u\right\vert
^{4}d\mu +\left( n-2-\lambda \right) \alpha \int_{M}Ju^{\alpha -1}\left\vert
\nabla u\right\vert ^{2}d\mu  \notag \\
&&-\left( 4-\lambda \right) \alpha \int_{M}u^{\alpha -1}A\left( \nabla
u,\nabla u\right) d\mu +\frac{n-4}{2}\int_{M}\left( Q-\lambda \sigma
_{2}\left( A\right) \right) u^{\alpha +1}d\mu .  \notag
\end{eqnarray}%
If%
\begin{equation}
\alpha \left( \alpha -1\right) +\frac{3\alpha -1}{2\left( n-4\right) }%
\lambda >0,  \label{eq3.8}
\end{equation}%
(this happens when $\left\vert \alpha \right\vert $ is large enough), then
using (\ref{eq2.9}) we get%
\begin{eqnarray}
&&\frac{n-4}{2}\int_{M}\chi u^{\alpha }d\mu  \label{eq3.9} \\
&\leq &\alpha \int_{M}u^{\alpha -1}\left( \Delta u\right) ^{2}d\mu -\left[ 
\frac{4-\lambda }{2\left( n-4\right) }\alpha ^{2}+\frac{16-2\lambda -n\left(
4-\lambda \right) }{2\left( n-4\right) ^{2}}\alpha \right] \int_{M}u^{\alpha
-3}\left\vert \nabla u\right\vert ^{4}d\mu  \notag \\
&&+\left( \frac{n-4}{2}\alpha ^{2}+\frac{2n-\lambda }{4}\alpha -\frac{%
\lambda }{4}\right) \int_{M}Ju^{\alpha -1}\left\vert \nabla u\right\vert
^{2}d\mu -\left( 4-\lambda \right) \alpha \int_{M}u^{\alpha -1}A\left(
\nabla u,\nabla u\right) d\mu  \notag \\
&&+\frac{n-4}{2}\int_{M}\left( Q-\lambda \sigma _{2}\left( A\right) \right)
u^{\alpha +1}d\mu .  \notag
\end{eqnarray}%
Now let $\alpha =-\left( p+1\right) $ with $p\gg 1$, using $\lambda \leq
\lambda _{0}<4$ and the fact for all $\varepsilon >0$,%
\begin{equation*}
u^{-p-2}\left\vert \nabla u\right\vert ^{2}\leq \varepsilon
u^{-p-4}\left\vert \nabla u\right\vert ^{4}+\frac{1}{4\varepsilon }u^{-p},
\end{equation*}%
we get%
\begin{eqnarray}
\int_{M}u^{-p-1}d\mu &\leq &-c\int_{M}u^{-p-2}\left( \Delta u\right)
^{2}d\mu -c\int_{M}u^{-p-4}\left\vert \nabla u\right\vert ^{4}d\mu
\label{eq3.10} \\
&&+c\int_{M}u^{-p}d\mu .  \notag
\end{eqnarray}%
Hence%
\begin{equation}
\left\Vert u^{-1}\right\Vert _{L^{p+1}}^{p+1}\leq c\left\Vert
u^{-1}\right\Vert _{L^{p}}^{p}\leq c\left\Vert u^{-1}\right\Vert
_{L^{p+1}}^{p}.  \label{eq3.11}
\end{equation}%
It follows that%
\begin{equation}
\left\Vert u^{-1}\right\Vert _{L^{p+1}}\leq c.  \label{eq3.12}
\end{equation}%
To continue let $u^{-1}=U$, then the inequality (\ref{eq2.9}) implies%
\begin{eqnarray}
-\Delta U &<&-\frac{2\left( n-3\right) }{n-4}U^{-1}\left\vert \nabla
U\right\vert ^{2}+\frac{n-4}{2}JU  \label{eq3.13} \\
&\leq &\frac{n-4}{2}JU.  \notag
\end{eqnarray}%
By \cite[Theorem 8.17 on p194]{GT}, (\ref{eq3.12}) and (\ref{eq3.13})
together imply $U\leq c$, in another word%
\begin{equation}
u\geq c>0.  \label{eq3.14}
\end{equation}
\end{proof}

To derive further estimates on $u$, we denote%
\begin{equation}
v=u^{-q-1}\left( -\Delta u+\frac{n-4}{2}Ju\right) ,  \label{eq3.15}
\end{equation}%
here $q\geq 0$ is a number to be determined. It follows from (\ref{eq2.9})
that%
\begin{equation}
v>\frac{2}{n-4}u^{-q-2}\left\vert \nabla u\right\vert ^{2}.  \label{eq3.16}
\end{equation}%
Since%
\begin{equation}
\Delta u=\frac{n-4}{2}Ju-u^{q+1}v,  \label{eq3.17}
\end{equation}%
we see that under local orthonormal frame with respect to $g$,%
\begin{eqnarray}
\Delta ^{2}u &=&-u^{q+1}\Delta v-2\left( q+1\right) u^{q}u_{i}v_{i}+\left(
q+1\right) u^{2q+1}v^{2}  \label{eq3.18} \\
&&-q\left( q+1\right) u^{q-1}\left\vert \nabla u\right\vert ^{2}v-\frac{n-4}{%
2}\left( q+2\right) Ju^{q+1}v  \notag \\
&&+\left( n-4\right) J_{i}u_{i}+\frac{\left( n-4\right) ^{2}}{4}J^{2}u+\frac{%
n-4}{2}\Delta J\cdot u.  \notag
\end{eqnarray}%
Plug (\ref{eq3.17}) and (\ref{eq3.18}) into (\ref{eq2.8}) we get%
\begin{eqnarray}
&&-\Delta v-2\left( q+1\right) u^{-1}u_{i}v_{i}+\left( q+1\right)
u^{q}v^{2}-q\left( q+1\right) u^{-2}\left\vert \nabla u\right\vert ^{2}v
\label{eq3.19} \\
&&+\left( -\frac{n-4}{2}q+2\right)
Jv+4u^{-q-1}A_{ij}u_{ij}+2u^{-q-1}J_{i}u_{i}-\left( n-4\right) \left\vert
A\right\vert ^{2}u^{-q}  \notag \\
&=&\frac{\lambda }{n-4}u^{q}v^{2}-\frac{2\lambda }{\left( n-4\right) ^{2}}%
u^{-2}\left\vert \nabla u\right\vert ^{2}v-\frac{\lambda }{n-4}%
u^{-q-2}\left\vert D^{2}u\right\vert ^{2}+\frac{2\left( n-2\right) \lambda }{%
\left( n-4\right) ^{2}}u^{-q-3}u_{ij}u_{i}u_{j}  \notag \\
&&+\lambda u^{-q-1}A_{ij}u_{ij}-\frac{\left( n-1\right) \lambda }{\left(
n-4\right) ^{2}}u^{-q-4}\left\vert \nabla u\right\vert ^{4}-\frac{\left(
n-2\right) \lambda }{n-4}u^{-q-2}A_{ij}u_{i}u_{j}-\frac{\left( n-4\right)
\lambda }{4}\left\vert A\right\vert ^{2}u^{-q}  \notag \\
&&+\frac{n-4}{2}\chi u^{-q-1}.  \notag
\end{eqnarray}%
Multiplying $v^{\alpha }$ on both sides and doing integration by parts we
have%
\begin{eqnarray}
&&\alpha v^{\alpha -1}\left\vert \nabla v\right\vert ^{2}+\frac{\left(
\alpha -1\right) \left( q+1\right) }{\alpha +1}u^{q}v^{\alpha +2}-\left(
q+1\right) \left( q+\frac{2}{\alpha +1}\right) u^{-2}\left\vert \nabla
u\right\vert ^{2}v^{\alpha +1}  \label{eq3.20} \\
&&-4\alpha u^{-q-1}A_{ij}u_{i}v_{j}v^{\alpha -1}+4\left( q+1\right)
u^{-q-2}A_{ij}u_{i}u_{j}v^{\alpha }+\left[ -\frac{\left( n-4\right) \left(
\alpha -1\right) }{2\left( \alpha +1\right) }q+\frac{2\alpha +n-2}{\alpha +1}%
\right] Jv^{\alpha +1}  \notag \\
&&-2u^{-q-1}J_{i}u_{i}v^{\alpha }-\left( n-4\right) \left\vert A\right\vert
^{2}u^{-q}v^{\alpha }  \notag \\
&\sim &\frac{\lambda }{n-4}u^{q}v^{\alpha +2}-\frac{2\lambda }{\left(
n-4\right) ^{2}}u^{-2}\left\vert \nabla u\right\vert ^{2}v^{\alpha +1}-\frac{%
\lambda }{n-4}u^{-q-2}\left\vert D^{2}u\right\vert ^{2}v^{\alpha }+\frac{%
2\left( n-2\right) \lambda }{\left( n-4\right) ^{2}}%
u^{-q-3}u_{ij}u_{i}u_{j}v^{\alpha }  \notag \\
&&-\frac{\left( n-1\right) \lambda }{\left( n-4\right) ^{2}}%
u^{-q-4}\left\vert \nabla u\right\vert ^{4}v^{\alpha }-\lambda \alpha
u^{-q-1}A_{ij}u_{i}v_{j}v^{\alpha -1}+\lambda \left( q-\frac{2}{n-4}\right)
u^{-q-2}A_{ij}u_{i}u_{j}v^{\alpha }  \notag \\
&&-\lambda u^{-q-1}J_{i}u_{i}v^{\alpha }-\frac{\left( n-4\right) \lambda }{4}%
\left\vert A\right\vert ^{2}u^{-q}v^{\alpha }+\frac{n-4}{2}\chi
u^{-q-1}v^{\alpha }.  \notag
\end{eqnarray}%
Here we write $\Phi \sim \Psi $ to mean $\int_{M}\Phi d\mu =\int_{M}\Psi
d\mu $. In view of (\ref{eq3.2}) and (\ref{eq3.16}) we see%
\begin{eqnarray}
&&\alpha \int_{M}v^{\alpha -1}\left\vert \nabla v\right\vert ^{2}d\mu +\left[
\frac{\left( \alpha -1\right) \left( q+1\right) }{\alpha +1}-\frac{\lambda }{%
n-4}\right] \int_{M}u^{q}v^{\alpha +2}d\mu  \label{eq3.21} \\
&\leq &\left[ \left( q+1\right) \left( q+\frac{2}{\alpha +1}\right) -\frac{%
2\lambda }{\left( n-4\right) ^{2}}\right] \int_{M}u^{-2}\left\vert \nabla
u\right\vert ^{2}v^{\alpha +1}d\mu -\frac{\lambda }{n-4}\int_{M}u^{-q-2}%
\left\vert D^{2}u\right\vert ^{2}v^{\alpha }d\mu  \notag \\
&&+\frac{2\left( n-2\right) \lambda }{\left( n-4\right) ^{2}}%
\int_{M}u^{-q-3}u_{ij}u_{i}u_{j}v^{\alpha }d\mu -\frac{\left( n-1\right)
\lambda }{\left( n-4\right) ^{2}}\int_{M}u^{-q-4}\left\vert \nabla
u\right\vert ^{4}v^{\alpha }d\mu  \notag \\
&&+c\alpha \int_{M}v^{\alpha -\frac{1}{2}}\left\vert \nabla v\right\vert
d\mu +c\int_{M}v^{\alpha +1}d\mu +c\int_{M}v^{\alpha }d\mu .  \notag
\end{eqnarray}%
To continue we split $D^{2}u$ into its trace component $\frac{\Delta u}{n}g$
and trace-free component $\Theta $,%
\begin{equation}
D^{2}u=\Theta +\frac{\Delta u}{n}g.  \label{eq3.22}
\end{equation}%
Consequently by (\ref{eq3.17}),%
\begin{eqnarray}
\left\vert D^{2}u\right\vert ^{2} &=&\left\vert \Theta \right\vert ^{2}+%
\frac{1}{n}u^{2q+2}v^{2}-\frac{n-4}{n}Ju^{q+2}v+\frac{\left( n-4\right) ^{2}%
}{4n}J^{2}u^{2},  \label{eq3.23} \\
u_{ij}u_{i}u_{j} &=&-\frac{1}{n}u^{q+1}\left\vert \nabla u\right\vert ^{2}v+%
\frac{n-4}{2n}Ju\left\vert \nabla u\right\vert ^{2}+\Theta _{ij}u_{i}u_{j}.
\label{eq3.24}
\end{eqnarray}%
Plug these equalities into (\ref{eq3.21}), we get%
\begin{eqnarray}
&&\alpha \int_{M}v^{\alpha -1}\left\vert \nabla v\right\vert ^{2}d\mu +\left[
\frac{\left( \alpha -1\right) \left( q+1\right) }{\alpha +1}-\frac{\left(
n-1\right) \lambda }{n\left( n-4\right) }\right] \int_{M}u^{q}v^{\alpha
+2}d\mu  \label{eq3.25} \\
&\leq &\left[ \left( q+1\right) \left( q+\frac{2}{\alpha +1}\right) -\frac{%
4\left( n-1\right) \lambda }{n\left( n-4\right) ^{2}}\right]
\int_{M}u^{-2}\left\vert \nabla u\right\vert ^{2}v^{\alpha +1}d\mu  \notag \\
&&-\frac{\lambda }{n-4}\int_{M}u^{-q-2}v^{\alpha }\left( \left\vert \Theta
\right\vert ^{2}-\frac{2\left( n-2\right) }{n-4}u^{-1}\Theta _{ij}u_{i}u_{j}+%
\frac{n-1}{n-4}u^{-2}\left\vert \nabla u\right\vert ^{4}\right) d\mu  \notag
\\
&&+\frac{\left( n-2\right) \lambda }{n\left( n-4\right) }\int_{M}Ju^{-q-2}%
\left\vert \nabla u\right\vert ^{2}v^{\alpha }d\mu -\frac{\left( n-4\right)
\lambda }{4n}\int_{M}J^{2}u^{-q}v^{\alpha }d\mu  \notag \\
&&+c\alpha \int_{M}v^{\alpha -\frac{1}{2}}\left\vert \nabla v\right\vert
d\mu +c\int_{M}v^{\alpha +1}d\mu +c\int_{M}v^{\alpha }d\mu .  \notag
\end{eqnarray}%
By (\ref{eq3.16}) we have%
\begin{equation}
\frac{\left( n-2\right) \lambda }{n\left( n-4\right) }\int_{M}Ju^{-q-2}\left%
\vert \nabla u\right\vert ^{2}v^{\alpha }d\mu \leq c\int_{M}v^{\alpha
+1}d\mu .  \label{eq3.26}
\end{equation}%
This and%
\begin{equation}
c\alpha \int_{M}v^{\alpha -\frac{1}{2}}\left\vert \nabla v\right\vert d\mu
\leq \frac{\alpha }{2}\int_{M}v^{\alpha -1}\left\vert \nabla v\right\vert
^{2}d\mu +\frac{c^{2}\alpha }{2}\int_{M}v^{\alpha }  \label{eq3.27}
\end{equation}%
together with (\ref{eq3.25}) implies for $\alpha \geq 1$,%
\begin{eqnarray}
&&\frac{\alpha }{2}\int_{M}v^{\alpha -1}\left\vert \nabla v\right\vert
^{2}d\mu +\left[ \frac{\left( \alpha -1\right) \left( q+1\right) }{\alpha +1}%
-\frac{\left( n-1\right) \lambda }{n\left( n-4\right) }\right]
\int_{M}u^{q}v^{\alpha +2}d\mu  \label{eq3.28} \\
&\leq &\left[ \left( q+1\right) \left( q+\frac{2}{\alpha +1}\right) -\frac{%
4\left( n-1\right) \lambda }{n\left( n-4\right) ^{2}}\right]
\int_{M}u^{-2}\left\vert \nabla u\right\vert ^{2}v^{\alpha +1}d\mu  \notag \\
&&-\frac{\lambda }{n-4}\int_{M}u^{-q-2}v^{\alpha }\left( \left\vert \Theta
\right\vert ^{2}-\frac{2\left( n-2\right) }{n-4}u^{-1}\Theta _{ij}u_{i}u_{j}+%
\frac{n-1}{n-4}u^{-2}\left\vert \nabla u\right\vert ^{4}\right) d\mu  \notag
\\
&&+c\int_{M}v^{\alpha +1}d\mu +c\alpha \int_{M}v^{\alpha }d\mu .  \notag
\end{eqnarray}%
Because $\Theta $ is trace-free,%
\begin{equation}
\left\vert \Theta _{ij}u_{i}u_{j}\right\vert \leq \sqrt{\frac{n-1}{n}}%
\left\vert \Theta \right\vert \left\vert \nabla u\right\vert ^{2},
\label{eq3.29}
\end{equation}%
hence%
\begin{eqnarray}
&&\left\vert \Theta \right\vert ^{2}-\frac{2\left( n-2\right) }{n-4}%
u^{-1}\Theta _{ij}u_{i}u_{j}+\frac{n-1}{n-4}u^{-2}\left\vert \nabla
u\right\vert ^{4}  \label{eq3.30} \\
&\geq &\left\vert \Theta \right\vert ^{2}-\frac{2\left( n-2\right) }{n-4}%
\sqrt{\frac{n-1}{n}}u^{-1}\left\vert \Theta \right\vert \left\vert \nabla
u\right\vert ^{2}+\frac{n-1}{n-4}u^{-2}\left\vert \nabla u\right\vert ^{4} 
\notag \\
&=&\left( \left\vert \Theta \right\vert -\frac{n-2}{n-4}\sqrt{\frac{n-1}{n}}%
u^{-1}\left\vert \nabla u\right\vert ^{2}\right) ^{2}-\frac{4\left(
n-1\right) }{n\left( n-4\right) ^{2}}u^{-2}\left\vert \nabla u\right\vert
^{4}  \notag \\
&\geq &-\frac{4\left( n-1\right) }{n\left( n-4\right) ^{2}}u^{-2}\left\vert
\nabla u\right\vert ^{4}.  \notag
\end{eqnarray}%
Plug this inequality into (\ref{eq3.28}) we get%
\begin{eqnarray}
&&\frac{\alpha }{2}\int_{M}v^{\alpha -1}\left\vert \nabla v\right\vert
^{2}d\mu  \label{eq3.31} \\
&&+\left[ \frac{\left( \alpha -1\right) \left( q+1\right) }{\alpha +1}-\frac{%
\left( n-1\right) \lambda }{n\left( n-4\right) }\right] \int_{M}u^{q}v^{%
\alpha +2}d\mu  \notag \\
&\leq &\left[ \left( q+1\right) \left( q+\frac{2}{\alpha +1}\right) -\frac{%
4\left( n-1\right) \lambda }{n\left( n-4\right) ^{2}}\right]
\int_{M}u^{-2}\left\vert \nabla u\right\vert ^{2}v^{\alpha +1}d\mu  \notag \\
&&+\frac{4\left( n-1\right) \lambda }{n\left( n-4\right) ^{3}}%
\int_{M}u^{-q-4}\left\vert \nabla u\right\vert ^{4}v^{\alpha }d\mu
+c\int_{M}v^{\alpha +1}d\mu +c\alpha \int_{M}v^{\alpha }d\mu .  \notag
\end{eqnarray}%
By (\ref{eq3.16}) we have%
\begin{equation}
\frac{4\left( n-1\right) \lambda }{n\left( n-4\right) ^{3}}%
\int_{M}u^{-q-4}\left\vert \nabla u\right\vert ^{4}v^{\alpha }d\mu \leq 
\frac{2\left( n-1\right) \lambda }{n\left( n-4\right) ^{2}}%
\int_{M}u^{-2}\left\vert \nabla u\right\vert ^{2}v^{\alpha +1}d\mu ,
\label{eq3.32}
\end{equation}%
hence%
\begin{eqnarray}
&&\frac{\alpha }{2}\int_{M}v^{\alpha -1}\left\vert \nabla v\right\vert
^{2}d\mu  \label{eq3.33} \\
&&+\left[ \frac{\left( \alpha -1\right) \left( q+1\right) }{\alpha +1}-\frac{%
\left( n-1\right) \lambda }{n\left( n-4\right) }\right] \int_{M}u^{q}v^{%
\alpha +2}d\mu  \notag \\
&\leq &\left[ \left( q+1\right) \left( q+\frac{2}{\alpha +1}\right) -\frac{%
2\left( n-1\right) \lambda }{n\left( n-4\right) ^{2}}\right]
\int_{M}u^{-2}\left\vert \nabla u\right\vert ^{2}v^{\alpha +1}d\mu  \notag \\
&&+c\int_{M}v^{\alpha +1}d\mu +c\alpha \int_{M}v^{\alpha }d\mu .  \notag
\end{eqnarray}%
By (\ref{eq3.16}) again,%
\begin{equation}
\int_{M}u^{-2}\left\vert \nabla u\right\vert ^{2}v^{\alpha +1}d\mu \leq 
\frac{n-4}{2}\int_{M}u^{q}v^{\alpha +2}d\mu ,  \label{eq3.34}
\end{equation}%
we get%
\begin{eqnarray}
&&\frac{\alpha }{2}\int_{M}v^{\alpha -1}\left\vert \nabla v\right\vert
^{2}d\mu  \label{eq3.35} \\
&&+\left[ \frac{\left( \alpha -1\right) \left( q+1\right) }{\alpha +1}-\frac{%
\left( n-1\right) \lambda }{n\left( n-4\right) }\right] \int_{M}u^{q}v^{%
\alpha +2}d\mu  \notag \\
&\leq &\max \left\{ \frac{n-4}{2}\left( q+1\right) \left( q+\frac{2}{\alpha
+1}\right) -\frac{\left( n-1\right) \lambda }{n\left( n-4\right) },0\right\}
\int_{M}u^{q}v^{\alpha +2}d\mu  \notag \\
&&+c\int_{M}v^{\alpha +1}d\mu +c\alpha \int_{M}v^{\alpha }d\mu .  \notag
\end{eqnarray}%
In another word%
\begin{eqnarray}
&&\frac{\alpha }{2}\int_{M}v^{\alpha -1}\left\vert \nabla v\right\vert
^{2}d\mu  \label{eq3.36} \\
&&+\min \left\{ \frac{\left( \alpha -1\right) \left( q+1\right) }{\alpha +1}-%
\frac{\left( n-1\right) \lambda }{n\left( n-4\right) },\left( q+1\right)
\left( -\frac{n-4}{2}q+\frac{\alpha -n+3}{\alpha +1}\right) \right\}
\int_{M}u^{q}v^{\alpha +2}d\mu  \notag \\
&\leq &c\int_{M}v^{\alpha +1}d\mu +c\alpha \int_{M}v^{\alpha }d\mu .  \notag
\end{eqnarray}%
Because $n\geq 6$, we see%
\begin{equation}
\frac{n-2}{n-4}>\frac{4\left( n-1\right) }{n\left( n-4\right) }.
\label{eq3.37}
\end{equation}%
Fix a $q\geq 0$ such that%
\begin{equation}
\frac{n-2}{n-4}>q+1>\frac{4\left( n-1\right) }{n\left( n-4\right) },
\label{eq3.38}
\end{equation}%
then for $\alpha $ large enough,%
\begin{equation}
\alpha \int_{M}v^{\alpha -1}\left\vert \nabla v\right\vert ^{2}d\mu
+\int_{M}u^{q}v^{\alpha +2}d\mu \leq c\int_{M}v^{\alpha +1}d\mu +c\alpha
\int_{M}v^{\alpha }d\mu .  \label{eq3.39}
\end{equation}%
By (\ref{eq3.2}) and (\ref{eq3.39}) we get%
\begin{equation}
\left\Vert v\right\Vert _{L^{\alpha +2}}^{\alpha +2}\leq c\left( \alpha
\right) \left( \left\Vert v\right\Vert _{L^{\alpha +2}}^{\alpha
+1}+\left\Vert v\right\Vert _{L^{\alpha +2}}^{\alpha }\right) .
\label{eq3.40}
\end{equation}%
Hence 
\begin{equation}
\left\Vert v\right\Vert _{L^{\alpha +2}}\leq c\left( \alpha \right)
\label{eq3.41}
\end{equation}%
for $\alpha $ large enough. To continue, we observe that for $\alpha $ large
enough,%
\begin{equation}
\int_{M}\left\vert \nabla v^{\frac{\alpha +1}{2}}\right\vert ^{2}d\mu \leq
c\alpha \int_{M}v^{\alpha +1}d\mu +c\alpha ^{2}\int_{M}v^{\alpha }d\mu .
\label{eq3.42}
\end{equation}%
Hence%
\begin{eqnarray}
\left\Vert v^{\frac{\alpha +1}{2}}\right\Vert _{L^{\frac{2n}{n-2}}}^{2}
&\leq &c\left( \int_{M}\left\vert \nabla v^{\frac{\alpha +1}{2}}\right\vert
^{2}d\mu +\int_{M}v^{\alpha +1}d\mu \right)  \label{eq3.43} \\
&\leq &c\alpha \left\Vert v\right\Vert _{L^{\alpha +1}}^{\alpha +1}+c\alpha
^{2}\left\Vert v\right\Vert _{L^{\alpha +1}}^{\alpha }.  \notag
\end{eqnarray}%
Replacing $\alpha +1$ by $\alpha $ we see%
\begin{equation}
\left\Vert v\right\Vert _{L^{\kappa \alpha }}^{\alpha }\leq c\alpha
\left\Vert v\right\Vert _{L^{\alpha }}^{\alpha }+c\alpha ^{2}\left\Vert
v\right\Vert _{L^{\alpha }}^{\alpha -1}.  \label{eq3.44}
\end{equation}%
Here%
\begin{equation}
\kappa =\frac{n}{n-2}.  \label{eq3.45}
\end{equation}%
Let%
\begin{equation}
b_{\alpha }=\max \left\{ \left\Vert v\right\Vert _{L^{\alpha }},1\right\} ,
\label{eq3.46}
\end{equation}%
then for $\alpha $ large enough, we have%
\begin{equation}
b_{\kappa \alpha }^{\alpha }\leq c\alpha ^{2}b_{\alpha }^{\alpha }.
\label{eq3.47}
\end{equation}%
It follows from iteration that for a fixed $\alpha _{0}$ large,%
\begin{equation}
b_{\kappa ^{k}\alpha _{0}}\leq cb_{\alpha _{0}}\leq c.  \label{eq3.48}
\end{equation}%
Hence%
\begin{equation}
\left\Vert v\right\Vert _{L^{\kappa ^{k}\alpha _{0}}}\leq c.  \label{eq3.49}
\end{equation}%
Letting $k\rightarrow \infty $ we see%
\begin{equation}
\left\Vert v\right\Vert _{L^{\infty }}\leq c.  \label{eq3.50}
\end{equation}%
Now we go back to the equation of $u$,%
\begin{equation}
-\Delta u+\frac{n-4}{2}Ju=vu^{q+1}.  \label{eq3.51}
\end{equation}%
Since $\left\Vert u\right\Vert _{L^{\frac{2n}{n-4}}}\leq c$ and $q<\frac{4}{%
n-4}$, standard bootstrap method tells us%
\begin{equation}
\left\Vert u\right\Vert _{L^{\infty }}\leq c.  \label{eq3.52}
\end{equation}%
Elliptic estimate tells us%
\begin{equation}
\left\Vert u\right\Vert _{W^{2,p}}\leq c\left( p\right)  \label{eq3.53}
\end{equation}%
for any $1<p<\infty $. This together with (\ref{eq2.8}) and (\ref{eq3.2})
imply%
\begin{equation}
\left\Vert u\right\Vert _{C^{k}}\leq c\left( k\right)  \label{eq3.54}
\end{equation}%
for all $k\in \mathbb{N}$.

\begin{proposition}
\label{prop3.1}Under the assumption of Theorem \ref{thm1.1}, the set $\Sigma 
$ is closed.
\end{proposition}

\begin{proof}
Assume $\lambda _{i}\in \Sigma $, $\lambda _{i}\rightarrow \lambda _{\infty
} $ as $i\rightarrow \infty $, $u_{i}$ is a solution to (\ref{eq2.4}) for $%
\lambda =\lambda _{i}$ satisfying $J_{i}>0$, then%
\begin{equation*}
\left\Vert u_{i}\right\Vert _{C^{k}}\leq c\left( k\right) ,
\end{equation*}%
for all $k\in \mathbb{N}$ and%
\begin{equation*}
u_{i}\geq c>0.
\end{equation*}%
After passing to a subsequence we have $u_{i}\rightarrow u_{\infty }$ in $%
C^{\infty }$ and $u_{\infty }\geq c>0$. Denote%
\begin{equation}
g_{i}=u_{i}^{\frac{4}{n-4}}g,\quad g_{\infty }=u_{\infty }^{\frac{4}{n-4}}g.
\label{eq3.55}
\end{equation}%
Then%
\begin{equation*}
Q_{i}-\lambda _{i}\sigma _{2}\left( A_{i}\right) =\chi u_{i}^{-\frac{n+4}{n-4%
}},\quad J_{i}>0.
\end{equation*}%
Let $i\rightarrow \infty $, we get%
\begin{equation}
Q_{\infty }-\lambda _{\infty }\sigma _{2}\left( A_{\infty }\right) =\chi
u_{\infty }^{-\frac{n+4}{n-4}},\quad J_{\infty }\geq 0.  \label{eq3.56}
\end{equation}%
In another way,%
\begin{equation}
-\Delta _{\infty }J_{\infty }-\frac{4-\lambda _{\infty }}{2}\left\vert
A_{\infty }\right\vert _{\infty }^{2}+\frac{n-\lambda _{\infty }}{2}%
J_{\infty }^{2}=\chi u_{\infty }^{-\frac{n+4}{n-4}}>0.  \label{eq3.57}
\end{equation}%
Hence%
\begin{equation}
-\Delta _{\infty }J_{\infty }+\frac{n-\lambda _{\infty }}{2}J_{\infty
}^{2}>0.  \label{eq3.58}
\end{equation}%
Since $J_{\infty }\geq 0$, it follows from strong maximum principle that
either $J_{\infty }>0$ or $J_{\infty }\equiv 0$. The latter case contradicts
with (\ref{eq3.58}). Hence $J_{\infty }>0$ and $\lambda _{\infty }\in \Sigma 
$. It follows that $\Sigma $ is closed.
\end{proof}

\section{Appendix\label{sec4}}

Here we give an elementary proof of Proposition \ref{prop2.1}. Assume $%
Y\left( M,g\right) >0$, we can assume $R>0$. For $1<p<\frac{n+2}{n-2}$,
based on the compact embedding $H^{1}\left( M\right) \subset L^{p+1}\left(
M\right) $ we know there exists a positive smooth function $u$ satisfying%
\begin{equation}
Lu=u^{p}  \label{eq4.1}
\end{equation}%
(see \cite{LP}). Here $L$ is the conformal Laplacian operator (\ref{eq1.2}).
Let%
\begin{equation}
\widetilde{g}=u^{\frac{4}{n-2}}g,  \label{eq4.2}
\end{equation}%
then%
\begin{equation}
\widetilde{R}=u^{-\frac{n+2}{n-2}}Lu=u^{p-\frac{n+2}{n-2}}>0,  \label{eq4.3}
\end{equation}%
hence $\widetilde{J}>0$. Note that%
\begin{equation}
-\widetilde{\Delta }\widetilde{J}+\frac{n-4}{2}\widetilde{J}^{2}=\frac{1}{%
2\left( n-1\right) }\left( -\widetilde{\Delta }\widetilde{R}+\frac{n-4}{%
4\left( n-1\right) }\widetilde{R}^{2}\right) .  \label{eq4.4}
\end{equation}%
Since%
\begin{equation}
\widetilde{\Delta }\varphi =u^{-\frac{4}{n-2}}\Delta \varphi +2u^{-\frac{n+2%
}{n-2}}g\left( \nabla u,\nabla \varphi \right) ,  \label{eq4.5}
\end{equation}
using (\ref{eq4.1}) and (\ref{eq4.3}) we have%
\begin{eqnarray}
&&-\widetilde{\Delta }\widetilde{R}+\frac{n-4}{4\left( n-1\right) }%
\widetilde{R}^{2}  \label{eq4.6} \\
&=&-\left( u^{-\frac{4}{n-2}}\Delta \widetilde{R}+2u^{-\frac{n+2}{n-2}%
}g\left( \nabla u,\nabla \widetilde{R}\right) \right) +\frac{n-4}{4\left(
n-1\right) }\widetilde{R}^{2}  \notag \\
&=&\frac{n-2}{4\left( n-1\right) }\left( p-\frac{6}{n-2}\right) u^{2p-\frac{%
2\left( n+2\right) }{n-2}}+\frac{n-2}{4\left( n-1\right) }\left( \frac{n+2}{%
n-2}-p\right) Ru^{p-\frac{n+6}{n-2}}  \notag \\
&&+\left( \frac{n+2}{n-2}-p\right) \left( p-\frac{4}{n-2}\right) u^{p-\frac{%
3n+2}{n-2}}\left\vert \nabla u\right\vert ^{2}.  \notag
\end{eqnarray}%
We choose $p$ such that%
\begin{equation}
\max \left\{ 1,\frac{6}{n-2}\right\} <p<\frac{n+2}{n-2},  \label{eq4.7}
\end{equation}%
this is possible since $n\geq 5$. Then it follows from (\ref{eq4.4}) and (%
\ref{eq4.6}) that%
\begin{equation}
-\widetilde{\Delta }\widetilde{J}+\frac{n-4}{2}\widetilde{J}^{2}>0.
\label{eq4.8}
\end{equation}

\end{document}